\documentclass[11pt, a4paper, reqno]{amsart}

\usepackage[utf8]{inputenc}
\usepackage{amsmath,amssymb}
\usepackage{graphicx}
\usepackage{indentfirst,csquotes}
\usepackage{pdflscape} 
\usepackage{hhline} 
\usepackage{array} 
\usepackage{calrsfs}
\DeclareMathAlphabet{\pazocal}{OMS}{zplm}{m}{n}
\usepackage{latexsym,amsmath,amsfonts,amscd,amssymb, mathrsfs, amsthm}
\usepackage[english]{babel}
\usepackage{appendix}
\usepackage{mathtools}
\usepackage{booktabs}
\usepackage{natbib}
\usepackage{hyperref}
\hypersetup{
    colorlinks=true,
    linkcolor=blue,
    urlcolor=blue,
    citecolor=blue
}
\advance\oddsidemargin by -0.5cm
\advance\evensidemargin by -0.5cm
\advance\topmargin by -1.5cm
\def\mxl{\left[ \begin{array}}  \def\mxr{\end{array} \right]}
\def\detl{\left| \begin{array}}  \def\detr{\end{array} \right|}
\def\bmx{\mxl{cccc}} \def\emx{\mxr}
\def\mx#1{\mxl{cccc}#1\mxr}

\def\bc{\begin{center}}          \def\ec{\end{center}}
\def\ben{\begin{enumerate}}      \def\een{\end{enumerate}}
\def\beq{\begin{equation}}       \def\eequ{\end{equation}}
\def\beqn{\begin{eqnarray*}} \def\eeqn{\end{eqnarray*}}
\def\eqs{\vspace{-10pt}\begin{equation}} \def\eqe{\vspace{-5pt}\end{equation}}
\def\bal{\begin{align}}           \def\eal{\end{align}}
\def\ben{\begin{enumerate}}       \def\een{\end{enumerate}}
\def\bit{\begin{itemize}}           \def\eit{\end{itemize}}
\def\btabb{\begin{tabbing}}         \def\etabb{\end{tabbing}}
\def\btab{\begin{tabular}}          \def\etab{\end{tabular}}
\def\barr{\begin{array}}           \def\earr{\end{array}}
\def\earrb{\end{array} \right\}}

\def\FF{\mathbb{F}}

 \def\imp{\Rightarrow}

\newcommand{\ovl}[1]{\overline{#1}}

\def\bf{\textbf}

\def\0{\bf{0}}    

\font\Bbb=msbm10 
\def\A{\ifmmode{\mbox{\Bbb A}}\else{{\Bbb A}}\fi}
\def\B{\ifmmode{\mbox{\Bbb B}}\else{{\Bbb B}}\fi}
\def\R{\ifmmode{\mbox{\Bbb R}}\else{{\Bbb R}}\fi}
\def\C{\ifmmode{\mbox{\Bbb C}}\else{{\Bbb C}}\fi}

\def\G{\ifmmode{\mbox{\Bbb G}}\else{{\Bbb G}}\fi}
\def\N{\ifmmode{\mbox{\Bbb N}}\else{{\Bbb N}}\fi}
\def\Z{\ifmmode{\mbox{\Bbb Z}}\else{{\Bbb Z}}\fi}
\def\K{\ifmmode{\mbox{\Bbb K}}\else{{\Bbb K}}\fi}
\def\L{\ifmmode{\mbox{\Bbb L}}\else{{\Bbb L}}\fi}
\def\D{\ifmmode{\mbox{\Bbb D}}\else{{\Bbb D}}\fi}
\def\T{\ifmmode{\mbox{\Bbb T}}\else{{\Bbb T}}\fi}
\def\I{\ifmmode{\mbox{\Bbb I}}\else{{\Bbb I}}\fi}
\def\E{\ifmmode{\mbox{\Bbb E}}\else{{\Bbb E}}\fi}
\def\K{\ifmmode{\mbox{\Bbb K}}\else{{\Bbb K}}\fi}
\def\Q{\ifmmode{\mbox{\Bbb Q}}\else{{\Bbb Q}}\fi}
\def\P{\ifmmode{\mbox{\Bbb P}}\else{{\Bbb P}}\fi}

\def\cG{\ifmmode{\mathcal{G}}\else{{$\mathcal{G}$}}\fi} 
\def\cA{\ifmmode{\mathcal{A}}\else{{$\mathcal{A}$}}\fi}

\def\op{ \ifmmode{\oplus} \else{\leavevmode\hbox{$\oplus$}\fi } }

\def\ot{\ifmmode{\otimes}\else{$\!\!\otimes$}\fi}
\def\od{ \ifmmode{\odot} \else{\leavevmode\hbox{$\odot$}\fi } }

\def\adots{\mathinner{\raise 1pt\hbox{.}\mkern1mu\raise4 pt\hbox{.}\mkern2mu
  \mkern1mu\raise7 pt\vbox{\kern7 pt\hbox{.}}\mkern2mu}}

\def\bdefn{\begin{defn}} \def\edefn{\end{defn}}
\def\bex{\begin{ex}} \def\eex{\end{ex}}
\def\beg{\begin{eg}} \def\eeg{\end{eg}}
\def\blem{\begin{lem}} \def\elem{\end{lem}}
\def\bth{\begin{thm}} \def\eth{\end{thm}}
\def\bprop{\begin{prop}} \def\eprop{\end{prop}}
\def\bcor{\begin{cor}} \def\ecor{\end{cor}}

\newcommand{\core}{\mathrel{^{\raisebox{0.6pt}{\scalebox{0.7}{\textcircled{\scriptsize $\sharp$}}}}}}

\newtheorem{thm}{Theorem}[section]  
\newtheorem{lem}{Lemma}[section]
\newtheorem{cor}{Corollary}[section]  \newtheorem{prop}{Proposition}[section]
\newtheorem{eg}{Example}[section]    
\newtheorem{ex}{Exercise}[section]
\newtheorem{defn}{Definition}[section]
 \numberwithin{equation}{section}

\def\ball{\begin{align}}  \def\eall{\end{align}}

\newcommand{\coreEP}{\mathrel{^{\raisebox{0.5pt}{\scalebox{0.6}{\textcircled{\scriptsize d}}}}}}

\def\bbin{\left( \begin{array}{c}} \def\ebin{\end{array} \right)}

\begin{document}

\title[Core-EP-Type Inverses for Structured Matrices]{Core EP, Dual Core EP and Composite Generalized Inverses for a Class of Structured Matrices}
\author[Faustino Maciala]{Faustino Maciala}
\address{Faustino Maciala, CMAT -- Centro de Matemática Universidade do Minho 4710-057 Braga Portugal;  Departamento de Ciências da Natureza e Ciências Exatas do Instituto Superior de Ciências da Educação de Cabinda -- ISCED-Cabinda, Angola.}
\email{fausmacialamath@hotmail.com}

\author[C. Mendes Araújo]{C. Mendes Araújo}
\address{C. Mendes Araújo, CMAT -- Centro de Matemática and Departamento de Matemática Universidade do Minho 4710-057 Braga Portugal}
\email{clmendes@math.uminho.pt}

\author[Pedro Patrício]{Pedro Patrício}
\address{Pedro Patrício, CMAT -- Centro de Matemática and Departamento de Matemática Universidade do Minho 4710-057 Braga Portugal}
\email{pedro@math.uminho.pt}


\date{\today}



\begin{abstract}
We study generalized inverses for matrices associated with double star digraphs. Explicit block formulas and existence criteria are obtained for core, dual core,
core EP, and dual core EP inverses, expressed in terms of explicit algebraic criteria derived from the underlying block structure. Other combined outer pseudoinverses, combining Moore--Penrose and core-type inverses, are derived with existence criteria.
\end{abstract} 

\subjclass[2020]{15A09, 05C20, 05C50}
\keywords{Drazin inverse, Moore--Penrose inverse, core-type inverses, core EP-type inverses, MPCEP inverse, $^\ast\!$CEPMP inverse, GDC inverse, GC inverse, digraph}

\maketitle

\section{Introduction}

Generalized inverses are standard tools for treating singular dynamics, constrained least squares problems, and spectral decompositions in linear and networked systems. The Drazin inverse $A^{D}$ extends inversion to arbitrary index, while the Moore--Penrose inverse $A^{\dagger}$ provides $*$-orthogonal projections and optimality properties. The reader is referred to \cite{benisraelgreville, campbell1979} for results on generalized inverses over the field of complex numbers, and to \cite{ChenBook} for a more recent approach to the algebraic theory of generalized inverses.

In graph matrix theory, the interaction between combinatorics and linear algebra has led to explicit formulas for inverses and generalized inverses of matrices associated with several graph families (trees, wheels, bipartite and path based classes), typically via full rank factorizations and Schur complements (\cite{Balaji,bookBapat,Bapat,Catral-grp,Catral2,Godsil,McLeman,McDonald}). 

A wide range of generalized inverses has been introduced in the literature as extensions of both the Moore--Penrose inverse and the Drazin inverse. These generalizations are typically obtained by modifying certain defining equations or characterization conditions, while selectively preserving, relaxing, or omitting others. Core-type generalized inverses and their variants have been extensively studied in rings with involution and in matrix/operator settings, including pseudo-core versions, core inverse for $2\times 2$ matrices and for certain products, and several recent refinements (\cite{Chen2020,gao2018pseudo,JiWei-coreEP2021,ke2019core,Mosic2024,MosaicMarovt-weighted-coreEP,Mosic2021,Stanimirovic,wang2019right,WuEtAl2024}). DMP and MPD inverses were introduced and analyzed by Malik and Thome (\cite{Malik2014}). 

We focus on a family of block-structured matrices associated with double star digraphs, where two hub-periphery components are coupled through their central vertices. Despite their simple form, these matrices admit a rich invertibility structure, allowing explicit characterizations and closed-form expressions for several generalized inverses via Schur complement and full rank factorization techniques.

In this paper we address generalized invertibility of matrices associated with \emph{double star digraphs}, namely we study:
(i) core-type inverses based on the existence of Moore--Penrose and group inverses;
(ii) core EP and dual core EP inverses using $(1,3)$-- and $(1,4)$--inverses criteria and full rank factorizations, with explicit block formulas; and
(iii) composite outer inverses, namely the Moore--Penrose--core EP inverse (MPCEP), the $^\ast\!$core EP--Moore--Penrose inverse ($^\ast\!$CEPMP), generalized dual core (GDC) inverse and  generalized core (GC) inverse, deriving block formulas and existence conditions for these compositions.

The paper is organized as follows. Section \ref{sec:prelim} covers preliminaries: double star digraphs and associated matrices, Drazin and Moore--Penrose inverses, $(1,3)$-- and $(1,4)$--inverses, and the Gram/Schur framework. Section \ref{sec:core} treats core and dual core inverses, with existence conditions and block formulas. Section \ref{sec:coreEP} develops core EP and dual core EP inverses, providing case-wise existence criteria and block expressions. Section \ref{sec:mpcep-cepmp} derives the compositions of Moore--Penrose and core-type inverses MPCEP and $^\ast\!$CEPMP, while Section \ref{sec:gdcdc-block} covers the generalized core and dual core inverses, both including existence conditions and block formulas. Finally, Section \ref{sec:conditions} provides a summary table collecting the existence conditions for all generalized inverses considered in this paper, restricted to the nontrivial cases in which the group inverse does not exist and the Drazin inverse is nonzero.

 \section{Preliminaries}\label{sec:prelim}

All our matrices are over a general field $\FF$ with involution $\ovl{\cdot}$. This involution on $\FF$ induces an involution $*$ on the set of finite matrices over $\FF$, defined as follows: if  $A=[a_{ij}],$ then  $A^*=[a_{ij}]^*=[\overline{a_{ji}}]$. Also, $\overline{A}=\overline{[a_{ij}]}=[\overline{a_{ij}}].$ 

Let $K_{1,m}$ and $K_{1,n}$ be two complete bipartite digraphs with central vertices $u$ and $v$, and vertex sets $\{u\}\cup U$, $U=\{u_1,\dots,u_m\}$, and $\{v\}\cup V$, $V=\{v_1,\dots,v_n\}$, respectively. The \emph{double star digraph} $S_{(m+1),(n+1)}$ is defined as the digraph obtained by joining the central vertices $u$ and $v$ with two directed edges, one in each direction. Equivalently, the vertex set of $S_{(m+1),(n+1)}$ is $U \cup V \cup \{u,v\}$, and it contains all directed edges of $K_{1,m}$ and $K_{1,n}$; that is, edges from $u$ to each $u_i$ and from each $u_i$ to $u$, and edges from $v$ to each $v_j$ and from each $v_j$ to $v$. In addition, the edges $u\to v$ and $v\to u$ are included, thus completing the structure of the double star digraph. An example of a double star digraph is shown in Figure \ref{fig:DSG}.

\begin{figure}[htbp]
\centering
\includegraphics[width=0.5\textwidth]{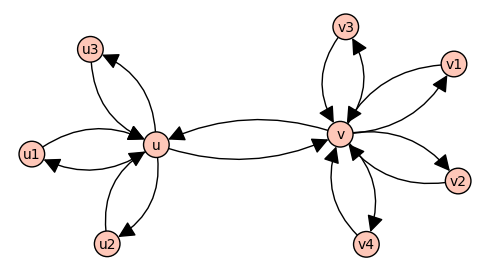}
\caption{Double star digraph $S_{(m+1),(n+1)}$ with $m=3$ leaves on the first star (center $u$) and $n=4$ leaves on the second star (center $v$).}
\label{fig:DSG}
\end{figure}

Given a square matrix $A = [a_{ij}]$ of order $n$, the \emph{associated digraph} $D(A) = (V,E)$ is defined to have vertex set $V = \{w_1,\ldots,w_n\}$ and edge set $E$, where $(w_i, w_j) \in E$ if and only if $a_{ij} \neq 0$. The entry $a_{ij}$ may be regarded as the weight of the directed edge from $w_i$ to $w_j$, and it is equal to zero precisely when such an edge is absent.

The block structure of any matrix associated with a double star digraph reflects the decomposition of the digraph into its two star components: each star contributes a diagonal block corresponding to the interactions between its central vertex and its pendant vertices, while the two directed edges joining the central vertices appear as off-diagonal blocks connecting these components.

Moreover, when different matrices are constructed from a double star digraph, their entries depend both on the chosen ordering of the vertices and on the weights assigned to its edges. However, whenever two such matrices correspond to the same weighted double star digraph -- that is, the edge weights are preserved -- they are permutation similar. Specifically, if $A$ and $B$ arise from two distinct vertex orderings but encode the same weighted digraph, then there exists a permutation matrix $P$ such that $A = P^{-1} B P$. We shall denote this relation by $A\underset{\text{perm.}}{\approx}B$. Since $P^{*} = P^{-1}$, this is a particular instance of unitary similarity. As a consequence, properties such as the existence of the Drazin inverse, the group inverse, the Moore--Penrose inverse, and the remaining pseudoinverses addressed in this study are independent of the ordering chosen for the vertices. 

\renewcommand{\arraystretch}{1.25}
Let $x$ and $y$ be vectors of length $m$, and $z$ and $w$ vectors of length $n$, each having all entries nonzero. Consider the square matrix of order $m+n+2$ defined by
\begin{equation}\label{eq:double-star-canonical}
M=  \mxl{cc|cc}
0 & x^T & a & 0\\
y & 0 & 0 & 0\\ \hline
b & 0 & 0 & z^T\\
0 & 0 & w & 0
\mxr,\end{equation}
where $a$ and $b$ are nonzero elements. Then the digraph $D(M)$ corresponds to a double star digraph $S_{(m+1),(n+1)}$ and any matrix $A$ with associated digraph $D(A)=S_{(m+1),(n+1)}$ is permutation similar to a matrix of the form described in \eqref{eq:double-star-canonical}.

For later use, we introduce the scalars  
\begin{equation}
\label{eq:all_scalar}
\begin{split}
& s := x^\ast x,\quad t := z^\ast z,\quad u := y^\ast y,\quad v := w^\ast w,\quad r := a\ovl{a} + w^\ast w, \\
& h := b\ovl{b} + y^\ast y,\quad p := b\ovl{b} + z^\ast z,\quad q := a\ovl{a} + x^\ast x,\quad \zeta := x^T y + ab, \\
& \beta := \zeta\ovl{\zeta} + b\ovl{b}\,w^\ast w,\quad  \alpha := \zeta\ovl{\zeta} + a\ovl{a}\,z^\ast z.
\end{split}
\end{equation}
These parameters play a central role in the Schur complement and Gram-type reductions that determine the rank behaviour and singularity patterns of matrices associated with double star digraphs. In particular, the special case $\zeta = 0$ induces structurally singular configurations in which several generalized inverse formulas simplify substantially.

As mentioned above, any matrix associated with a given double star digraph $S_{(m+1),(n+1)}$ is unitarily similar to a matrix in the canonical block form \eqref{eq:double-star-canonical}. Consequently, no generality is lost by performing the analysis exclusively on matrices of this form. In particular, all issues related to Moore--Penrose, Drazin, $(1,3)$--, and $(1,4)$--invertibility may be treated within this canonical representation.

The Drazin inverse $A^{D}$ of a square matrix $A$ over an arbitrary field is the unique matrix satisfying
\[
A^{\,k+1}A^{D}=A^{\,k}, \qquad
A^{D}AA^{D}=A^{D}, \qquad
AA^{D}=A^{D}A,
\]
where $k$ denotes the Drazin index of $A$, written $i(A)$. The Drazin index agrees with the algebraic index determined by the minimal polynomial. Equivalently, $i(A)$ is the smallest nonnegative integer $k$ such that
\[
\operatorname{rank}(A^{k})=\operatorname{rank}(A^{k+1}).
\]
The group inverse $A^{\#}$ is the special case of the Drazin inverse corresponding to matrices of index at most one, in Refs. \cite{Cline1965,Cline1968}.

In \cite{McDonald} the authors establish necessary and sufficient conditions for the existence of the group inverse of a matrix of the form \eqref{eq:double-star-canonical}. Specifically, $M^{\#}$ exists if and only if $x^{T}y \neq 0$ and $z^{T}w \neq 0$. In this case, the group inverse is explicitly given by
\begin{equation}
\label{eq:Msharp}
M^\sharp = 
\mxl{cc|cc}
0 & (x^T y)^{-1}x^T & 0 & 0 \\
(x^T y)^{-1}y & 0 & 0 & -a[(x^T y)(z^T w)]^{-1}yz^T \\ \hline
0 & 0 & 0 & (z^T w)^{-1}z^T \\
0 & -b[(x^T y)(z^T w)]^{-1}wx^T & (z^T w)^{-1}w & 0
\mxr.
\end{equation}

When $M$ fails to admit a group inverse (that is, when $x^{T}y = 0$ or $z^{T}w = 0$) the analysis must be carried out within the framework of Drazin invertibility, see \cite{Dragana}. In this situation, and without loss of generality, the study reduces to three mutually exclusive and exhaustive cases (see \cite{maciala2025characterization, araujo2026}), which constitute the structural basis for the pseudoinvertibility results developed in the subsequent sections. These cases are stated below.

\paragraph{Case I: $x^T y=0=z^T w.$}
Then $i(M)=2$ and
\begin{equation}
 \label{CaseI_Dinv}
M^{D}=(ab)^{-1}\!
\mxl{cc|cc}
0 & x^T & a & 0\\
y & 0 & 0 & b^{-1}yz^T\\ \hline
b & 0 & 0 & z^T\\
0 & a^{-1}wx^T & w & 0
\mxr.
\end{equation}

\paragraph{Case II: $x^T y \neq 0=z^T w$ \textnormal{and} $\zeta=x^T y + ab \neq  0.$}
Then $i(M)=3$ and
 \begin{equation}
 \label{CaseII_Dinv}
M^D = \zeta^{-1} \mxl{cc|cc}
0 & x^T  & a & 0\\
y & 0 & 0 & \zeta^{-1}ayz^T \\ \hline
b & 0 & 0 & \zeta^{-1}abz^T \\
0 & \zeta^{-1}bwx^T  & \zeta^{-1}abw & 0
\mxr.
\end{equation}

\paragraph{Case III: $x^T y \neq 0=z^T w$ \textnormal{and} $\zeta=x^T y + ab =0.$}
Then $i(M)=5$ and
\begin{equation}
\label{CaseIII_Dinv}
M^D = 0.
\end{equation}

It is worth noting that the situation in which $x^{T}y = 0$ and $z^{T}w \neq 0$ does not introduce a genuinely new configuration: by a suitable permutation similarity, this case reduces to the previously analysed situation where $x^{T}y \neq 0$ and $z^{T}w = 0$. Consequently, the three cases stated above provide a complete and mutually exclusive description of the Drazin invertibility of matrices associated with double star digraphs that do not admit a group inverse. 

Throughout the remainder of the paper, any reference to Case~I, Case~II, or Case~III is to be interpreted with respect to the above tripartite classification. 

A matrix $A$ is said to be Moore--Penrose invertible (with respect to $*$) if the system of equations
\begin{equation*}
\label{Eq:MoorePenrose}
\begin{aligned}
    &(1) \; A X A = A, \quad
    (2) \; XAX = X, \quad
    (3) \; (AX)^{*} = AX, \quad
    (4) \; (XA)^{*} = XA.
\end{aligned}
\end{equation*}
admits a common solution. When it exists, this solution is unique and is denoted by $A^{\dagger}$.

A matrix satisfying equation (1) is called a von Neumann inverse of $A$ and is denoted by $A^{-}$; the set of all von Neumann inverses of $A$ is written $A\{1\}$. A common solution of (1) and (2) is called a reflexive inverse of $A$, written $A^{+}$. For any two von Neumann inverses $A^{-},A^{=}\in A\{1\}$ of $A$, the product $A^{-} A A^{=}$ is a reflexive inverse of $A$.

We denote by $A\{1,3\}$ (respectively $A\{1,4\}$) the set of matrices satisfying equations (1) and (3) (respectively (1) and (4)); such matrices are called $(1,3)$--inverses (respectively $(1,4)$--inverses) of $A$. We write $A^{(1,3)}$ and $A^{(1,4)}$ for arbitrary elements of these sets whenever they are nonempty.

We now examine existence criteria for $(1,3)$-- and $(1,4)$--inverses. Classical results for $\ast$-regular rings appear in \cite{HartwigBlock}\label{13inv}, and the matrix case admits a transparent formulation in terms of full rank factorizations and column--space conditions, as stated in the following proposition \cite{Puystjens}.

We adopt the notation $\mathcal{R}(A)$ to denote the range of a matrix $A$.

\begin{prop}\label{eq:13-14-canonical}
 Given a full rank factorization of a matrix $A=FG$,
\ben
\item the following are equivalent:
\ben
\item $A\{1, 3\}\ne \varnothing$
\item $F^*F$ is invertible
\item $\mathcal{R}(A^*)=\mathcal{R}(A^*A)$
\item $\operatorname{rank}(A^*A)=\operatorname{rank}(A)$
\een
In this case, $(A^*A)^-A^*\in A\{1, 3\}$.

\item the following are equivalent:
\ben
\item $A\{1, 4\}\ne \varnothing$
\item $GG^*$ is invertible
\item $\mathcal{R}(A)=\mathcal{R}(AA^*)$
\item $\operatorname{rank}(AA^*)=\operatorname{rank}(A)$
\een
In this case, $A^*(AA^*)^-\in A\{1, 4\}$.

\een
\end{prop}
\proof We prove part~(1); the proof of part~(2) follows by a completely analogous argument.

$(1b)\imp (1a)$. We may take, as a reflexive inverse of $F$, any left inverse, that is, a matrix $F^{+}$ satisfying $F^{+}F = I$. Likewise, we may choose as a reflexive inverse of $G$ any right inverse $G^{+}$ such that $GG^{+} = I$.

Let $B$ and $C$ be matrices of compatible sizes. If $BA^{*}A = CA^{*}A$, then
\[
BG^{*}F^{*}FG = CG^{*}F^{*}FG,
\]
and hence
\[
BG^{*}F^{*}F = CG^{*}F^{*}F.
\]
Since $F^{*}F$ is invertible, we obtain $BG^{*} = CG^{*}$, and therefore $BA^{*} = CA^{*}$.

Analogously, if $A^*AB=A^*AC$ then $AB=AC$.

Applying this one-sided \(*\)-cancellation yields
\[
A(A^{*}A)^{-}A^{*}A = A,
\]
as it follows from
\[
A^{*}A (A^{*}A)^{-} A^{*}A = A^{*}A.
\]

It remains to show that $A(A^{*}A)^{-}A^{*}$ is symmetric. To this end, observe that $\left((A^*A)^-\right)^*\in A^*A\{1\}$.

The product $A(A^*A)^-A^*$ is invariant under the choice of von Neumann inverse of $A^*A$; that is,
\[
A(A^*A)^-A^*=A(A^*A)^=A^*
\]
for any $(A^*A)^-, (A^*A)^=\in A^*A\{1\}$. This follows directly from $A^*A(A^*A)^-A^*A=A^*A(A^*A)^=A^*A$ applying the $*$-cancellation. In particular, $A(A^*A)^-A^*=A\left((A^*A)^-\right)^*A^*$, and the symmetry of $A(A^*A)^-A^*$ follows.

$(1a)\imp (1c)$. Let $X\in A\{1, 3\}$. Then $A^*=A^*AX$, and hence $\mathcal{R}(A^*)\subseteq \mathcal{R}(A^*A)$.

$(1c)\imp (1b)$. If $\mathcal{R}(A^*)= \mathcal{R}(A^*A)$, then there exists a matrix $X$ such that $A=XA^*A$. Given a full rank factorization $A=FG$, set $F^+, G^+$ such that $F^+F=I, GG^+=I$. The identity $A = X A^{*}A$ becomes
\[
FG = XG^{*}F^{*}FG.
\]
Multiplying on the left by \(F^{+}\) and on the right by \(G^{+}\) yields
\[
I = (F^{+} X G^{*})(F^{*}F),
\]
so \(F^{*}F\) is invertible.

The equivalence $(1c) \Leftrightarrow (1d)$ is immediate from the rank identities.
\endproof

As an example, consider the complex matrix $A=\bmx 1\\ i\emx$, and  take the involution to be the transposition of matrices (that is, $*=T$). Since $\operatorname{rank}(A^*A)=0\ne 1=\operatorname{rank}(A)$, the matrix $A$ has no $(1, 3)$--inverse (and therefore it is not Moore--Penrose invertible).

We shall frequently require the Moore--Penrose invertibility of matrices associated with
double star digraphs. The following result, taken from Theorem 3.5 in \cite{araujo2026},
provides the precise existence criteria and the explicit formula in terms of the parameters
$s,u,t,v$ introduced in \eqref{eq:all_scalar}.
\begin{lem} \label{thmMPdoublestar} Let $M$ be a matrix  with associated double star digraph $S_{(m+1),(n+1)}$ of the form \eqref{eq:double-star-canonical}, with $x=[x_i], y = [y_i], z=[z_i]$ and $w=[w_i]$. Then
$M^\dagger$ exists if and only if $s,u,t,v\in \mathbb{F}\setminus \{0\}$, where $s = \sum x_i\ovl{x_i} , u=\sum y_i\ovl{y_i} , t=\sum z_i\ovl{z_i} , v= \sum w_i\ovl{w_i}$. In this case,
\begin{equation*}
\label{Moore_Penrose}
    M^\dagger = \mxl{cc|cc} 0 & u^{-1}y^* & 0 & 0 \\ s^{-1}\bar{x} & 0 & 0 & -s^{-1}av^{-1}\bar{x}w^* \\ \hline 0 & 0 & 0 &v^{-1}w^* \\ 0 & -t^{-1}bu^{-1}\bar{z}y^* & t^{-1}\bar{z} & 0 \mxr.
\end{equation*}
\end{lem}

The notion of \emph{core inverse} was introduced by Baksalary and Trenkler in \cite{baksalary2010core} as a refinement of the group inverse.
For a square matrix $A$, the core inverse of $A$ is the unique matrix  $A^{\core}$ satisfying
\begin{equation*}
AA^{\core} = P_{A}
\qquad \text{and} \qquad
\mathcal{R}(A^{\core}) \subseteq  \mathcal{R}(A),
\end{equation*}
where $P_{A}$ denotes the orthogonal projector onto $\mathcal{R}(A)$.
A matrix $A$ with Moore-Penrose inverse is core invertible if and only if $i(A) \le 1$; in this case, the core inverse exists and is given by
\begin{equation}\label{eq:formula_core_inverse}
A^{\core} = A^{\sharp} A A^{\dagger}.
\end{equation} 

The dual concept, the \emph{dual core inverse}, was introduced analogously in \cite{baksalary2010core}. It is defined as the unique matrix $A_{\core}$ satisfying
\begin{equation*}
A_{\core} A = P_{A^{*}}
\qquad \text{and} \qquad
\mathcal{R}(A_{\core}) \subseteq \mathcal{R}(A^{*}),
\end{equation*}
where $P_{A^{*}}$ projects onto $\mathcal{R}(A^{*})$.
Whenever $A$ is Moore-Penrose invertible and $i(A) \le 1$, the dual core inverse also exists and admits the representation
\begin{equation}\label{eq:formula_dual_core_inverse}
A_{\core} = A^{\dagger} A A^{\sharp}.
\end{equation} 

When studying matrices associated with double star digraphs, the explicit criteria for group invertibility and Moore--Penrose invertibility already established allow
the formulas~\eqref{eq:formula_core_inverse} and~\eqref{eq:formula_dual_core_inverse} to be applied directly.

We now introduce the notion of \emph{core EP invertibility}, also known as pseudo-core invertibility, as presented in \cite{gao2018pseudo}. A matrix $A$ is
said to be core EP invertible if there exists a matrix solution $X$ to
\begin{align*}
    \text{(i)} & \; XA^{m+1} = A^m, \; \text{for some nonnegative integer } m, \\
    \text{(ii)} & \; AX^2 = X, \\
    \text{(iii)}& \; (AX)^* = AX.
\end{align*}
In case a solution exists, it is unique and is called the core EP inverse of $A$, denoted by $A^{\coreEP}$.

Dually, $A$ is said to be dual core EP invertible if there is a matrix $X$ such that
\begin{align*}
    \text{(i')} & \; A^{m+1}X = A^m, \; \text{for some nonnegative integer } m, \\
    \text{(ii')} & \; X^2A = X, \\
     \text{(iii')} & \; (XA)^* = XA.
\end{align*}
in which case it is unique and is called the dual core EP inverse of $A$, denoted by $A_{_{\coreEP}}$.

When $m=1$, the (dual) core EP inverse coincides with the (dual) core inverse. Characterizations of this type of inverse and generalizations can be found in \cite{JiWei-coreEP2021,ke2019core, Li, wang2019right,WuEtAl2024}. Generalizations and extensions of core and core EP invertibility can be found for instance in \cite{Mosic2025, MosaicMarovt-weighted-coreEP,Mosic2021, Zhu2023}.

We shall use the following lemma that in particular relates core EPness (respectively dual core EPness) to Drazin and $(1,3)$--inverses (respectively $(1,4)$--inverses).
\begin{lem}[\cite{gao2018pseudo},Theorem 2.3] Given a square matrix  $A$ with Drazin index $i(A)$,
\label{CoreEP1,2}
\begin{enumerate}
    \item $A^{\coreEP}$ exists if and only if $A^D$ exists and $(A^m)\{1,3\} \neq \varnothing$, where $m\geq i(A)$. In this case, $A^{\coreEP} = A^DA^m(A^m)^{(1,3)}$ for any $(A^m)^{(1,3)} \in (A^m)\{1,3\}.$
    \item $A_{_{\coreEP}}$ exists if and only if $A^D$ exists and $(A^m)\{1,4\} \neq \varnothing$, where $m\geq i(A)$. In this case, $A_{_{\coreEP}} = (A^m)^{(1,4)}A^m A^D$ for any $(A^m)^{(1,4)} \in (A^m)\{1,4\}$.
\end{enumerate}
\end{lem}

We note that core EP invertibility is invariant under matrix unitary similarity. That is to say, if $B=UAU^*$ with $U^*=U^{-1}$ then $A$ is core EP invertible if and only if $B$ is core EP invertible. A similar result is obtained for dual core EP.

The notion of a composite generalized inverse obtained by successively
applying two generalized inverses was recently investigated in
\cite{Chen2020}, where the authors introduce the Moore--Penrose--core EP (MPCEP) inverse and the $^\ast\!$core EP--Moore--Penrose ($^\ast\!$CEPMP) inverse. 

Let $A$ be a square matrix of order $n$ and assume that the core EP inverse
$A^{\coreEP}$ and the Moore--Penrose inverse $A^\dagger$ exist. The \emph{Moore--Penrose--core EP inverse} of $A$ is the unique
matrix $X$ satisfying the system
\begin{equation*}
X A X = X, \qquad
A X = A A^{\coreEP}, \qquad
X A = A^{\dagger} A A^{\coreEP} A.
\end{equation*}
This system is consistent whenever $A^\dagger$ and $A^{\coreEP}$ exist,
and its unique solution is given explicitly by
\begin{equation}\label{eq:def-mpcep}
A^{\dagger,\coreEP} \;=\; A^{\dagger}\, A\, A^{\coreEP},
\end{equation}
which is called the \emph{MPCEP inverse} of $A$.

Assume now that the dual form of the core EP inverse, $A_{\coreEP}$, exists, as well as the Moore--Penrose inverse $A^\dagger$. The \emph{$^\ast\!$core EP--Moore--Penrose inverse} of $A$ is the unique
matrix $X$ satisfying
\begin{equation*}
X A X = X, \qquad
A X  = A A_{\coreEP} A A^{\dagger}, \qquad
X A = A_{\coreEP} A.
\end{equation*}
Whenever $A^\dagger$ and $A_{\coreEP}$ exist, the system is consistent
and its unique solution is
\begin{equation}\label{eq:def-cepmp}
A_{\coreEP,\dagger} \;=\; A_{\coreEP}\, A\, A^{\dagger},
\end{equation}
which is called the \emph{$^\ast\!$CEPMP inverse} of $A$.

By analogy with the MPCEP and $^\ast\!$CEPMP inverses, other composite generalized 
inverses can be formed by combining the Moore--Penrose inverse with the core EP 
inverse or with the dual core EP inverse.  The \emph{generalized dual core} (GDC) inverse 
and the \emph{generalized core} (GC) inverse were introduced in 
\cite{Mosic2024}. Over a general involutive field, the existence of each of these composite inverses requires the simultaneous existence of the corresponding core-type inverse 
and of the Moore--Penrose inverse. 

Let $A$ be a square matrix and assume that the core EP inverse 
$A^{\coreEP}$ and the Moore--Penrose inverse $A^\dagger$ exist.  
The \emph{generalized dual core inverse} of $A$, denoted by $A^{\mathrm{GDC}}$, is the unique matrix 
$X$ satisfying
\begin{equation*}
XAX = X, \qquad
AX = A^{\coreEP} A, \qquad
XA = A^{\dagger} A^{\coreEP} A^{2}.
\end{equation*}
Under these hypotheses this system is consistent, and its 
unique solution is given explicitly by
\begin{equation}\label{eq:gdc-explicit}
A^{\mathrm{GDC}} \;=\; A^{\dagger}\, A^{\coreEP}\, A.
\end{equation}

Assume now that $A$ is such that the dual core EP inverse 
$A_{\coreEP}$ and the Moore--Penrose inverse $A^\dagger$ exist.  
The \emph{generalized core inverse} of $A$, denoted by $A^{\mathrm{GC}}$, is the unique matrix 
$X$ satisfying
\begin{equation*}
XAX = X, \qquad
AX = A^{2} A_{\coreEP} A^{\dagger}, \qquad
XA = A A_{\coreEP}.
\end{equation*}
Whenever these conditions hold, this system is consistent, 
and its unique solution is
\begin{equation}\label{eq:gc-explicit}
A^{\mathrm{GC}} \;=\; A\, A_{\coreEP}\, A^{\dagger}.
\end{equation}


\section{Core and dual core invertibility}\label{sec:core}

A matrix admits a core inverse and a dual core inverse if and only if it admits
both a Moore--Penrose inverse and a group inverse. For matrices associated with double star digraphs, of the form \eqref{eq:double-star-canonical}, we have seen that this is equivalent to requiring that all Gram scalars $s,u,t,v$ and both coupling terms $x^Ty$ and $z^Tw$ are
nonzero. Under these assumptions, the explicit formulas for $M^\dagger$ and $M^\sharp$ 
given in Lemma \ref{thmMPdoublestar} and in \eqref{eq:Msharp} may be applied directly in the computation of $M^{\core}$ and $M_{\core}$. The orthogonal projectors $M^\dagger M$ and $M M^\dagger$ take the explicit block forms
\begin{equation}\label{projectors}
M^\dagger M
=\mxl{cccc}
1 & 0 & 0 & 0 \\[2pt]
0 & s^{-1}\,\bar{x} x^T & 0 & 0 \\[2pt]
0 & 0 & 1 & 0 \\[2pt]
0 & 0 & 0 & t^{-1}\,\bar{z} z^T
\mxr,
\qquad
M M^\dagger
=\mxl{cccc}
1 & 0 & 0 & 0 \\[2pt]
0 & u^{-1}\,y y^* & 0 & 0 \\[2pt]
0 & 0 & 1 & 0 \\[2pt]
0 & 0 & 0 & v^{-1}\,w w^*
\mxr.
\end{equation}
These identities are used in the computation of the core and dual core 
inverses of $M$.

\begin{prop} Let $M$ be a matrix with associated double star digraph $S_{(m+1),(n+1)}$ of the form \eqref{eq:double-star-canonical}, with $x=[x_i], y = [y_i], z=[z_i]$ and $w=[w_i]$. Assume $x^Ty\neq 0$, $z^Tw\neq 0$ and $s,u,t,v\in \mathbb{F}\setminus \{0\}$, where $s = \sum x_i\ovl{x_i} , u=\sum y_i\ovl{y_i} , t=\sum z_i\ovl{z_i} , v= \sum w_i\ovl{w_i}$. Then,
\begin{enumerate}
\item $M$ is core invertible and 
\[
M^{\core} = \mxl{cc|cc} 0 & u^{-1}y^\ast & 0 & 0 \\[2pt]
(x^T y)^{-1}y & 0 & 0 & -av^{-1}(x^T y)^{-1}yw^\ast \\[2pt] \hline
0 & 0 & 0 & v^{-1}w^\ast \\[2pt]
0 & -bu^{-1}(z^T w)^{-1} wy^\ast & (z^T w)^{-1} w & 0 \mxr;\]
\item $M$ is dual core invertible and 
\[M_{\core} = \mxl{cc|cc} 0 & (x^T y)^{-1}x^T & 0 & 0 \\[2pt]
s^{-1}\ovl{x} & 0 & 0 & -as^{-1}(z^T w )^{-1}\ovl{x} z^T \\[2pt] \hline 
0 & 0 & 0 & (z^T w)^{-1} z^T \\[2pt]
0 & -bt^{-1}(x^T y)^{-1} \ovl{z}x^T & t^{-1}\ovl{z} & 0
\mxr.
\]
\end{enumerate}
\end{prop}

This section completes the analysis of core and dual core invertibility for
matrices associated with double star digraphs. In this setting, core-type
invertibility is entirely determined by the simultaneous existence of the
Moore--Penrose inverse and the group inverse. In the next section we turn to matrices of higher index, for which the group inverse no longer exists and core EP-type inverses provide the appropriate framework.


\section{Core EP and dual core EP invertibility}\label{sec:coreEP}

We now turn to the study of core EP and dual core EP invertibility for matrices
associated with double star digraphs. Throughout this section, $M$ is assumed
to be of the canonical form \eqref{eq:double-star-canonical}  and to satisfy $i(M) > 1$.

We analyse separately the three mutually exclusive Cases I, II, and III described
in the preliminaries. In each case, the Drazin inverse $M^D$ is known explicitly
in block form.

First, consider Case I, where $x^Ty=z^Tw=0$.
In this case, $i(M)=2$ and the Drazin inverse of $M$ is given explicitly by \eqref{CaseI_Dinv}. A direct block multiplication yields
$$M^2 = \mxl{cc|cc}
ab & 0 & 0 & \mathit{az^T } \\
0& y\mathit{x^T } & ay & 0 \\
\hline
 0 & b\mathit{x^T } & ab & 0 \\
bw & 0 & 0 & w\mathit{z^T }
\mxr,$$
and therefore $$
\left(M^2 \right)^*= \mxl{cc|cc}
\ovl{a}\ovl{b} & 0 & 0 & \mathit{\ovl{b}w^*} \\
0& \ovl{x}y^* & \ovl{b}\ovl{x} & 0 \\
\hline
 0 & \ovl{a}y^* & \ovl{a}\ovl{b} & 0 \\
\bar{a}\bar{z} & 0 & 0 & \ovl{z}w^*
\mxr.$$

From Proposition \ref{eq:13-14-canonical}, we know that the existence of core EP and dual core EP inverses of $M$ depends on the existence of $(1,3)$-- and $(1,4)$--inverses of $M^2$, respectively. We analyze these two situations separately.

To study the existence of a $(1,3)$--inverse of $M^2$, we consider the Gram-type product $(M^2)^*M^2$. A straightforward computation gives
$$\left(M^2 \right)^* M^2
= \mxl{cc|cc} b\ovl{b}r & 0 & 0 & \ovl{b}rz^T  \\
0 & \ovl{x}hx^T  & a\ovl{x}h & 0 \\ \hline
0 & \ovl{a}hx^T  & a\ovl{a}h & 0 \\
b\ovl{z}r & 0 & 0 & \ovl{z}rz^T \mxr,$$
where, as in \eqref{eq:all_scalar}, $r= a\ovl{a}+w^{*}w$ and $h = b\ovl{b} + y^{*}y$. Up to permutation similarity, this matrix decomposes as a block diagonal matrix,
$$ \left(M^2 \right)^* M^2\underset{\text{perm.}}{\approx}\mxl{cc} rA & 0 \\ 0 & hB \mxr,$$
with $A=\mxl{cc} b\ovl{b} & \ovl{b}z^T  \\ b\bar{z} & \ovl{z}z^T  \mxr$ and $B=\mxl{cc}a\ovl{a} & \ovl{a}x^T  \\ a\ovl{x} & \ovl{x}x^T  \mxr.$

Now Proposition \ref{eq:13-14-canonical} asserts that the existence of a $(1,3)$--inverse of $M^2$ is equivalent to $\operatorname{rank}\left(\left(M^2 \right)^* M^2 \right) = \operatorname{rank}(M^2)=2  $, which in turn means $r, h\in \FF\setminus\{0\}.$ Assuming these conditions, $\left(\left(M^2 \right)^*  M^2  \right)^{-} \left(M^2 \right)^*\in M^2\{1, 3\}$. We therefore seek a von Neumann inverse of $\left(M^2 \right)^*M^2 $. We may take
 $$\left[\left(M^2 \right)^*\left(M^2 \right) \right]^{-} \underset{\text{perm.}}{\approx}  \mxl{cc} r^{-1}A^{-} & 0 \\ 0 & h^{-1}B^{-}\mxr. $$
 We next construct explicit von Neumann inverses for $A$ and $B$. Considering the full rank factorizations $$A=  \mx{\ovl{b} \\ \ovl{z}} \bmx b & z^T \emx = F_AG_A$$ and $$B= \mx{\ovl{a} \\ \ovl{x}} \bmx a & x^T \emx = F_BG_B,$$ and choosing von Neumman inverses 
 $$F_A^{-}= \bmx \ovl{b}^{-1} & 0\emx, \quad G_A^{-} = \bmx b^{-1} \\ 0\emx, \quad F_B^{-} = \bmx \ovl{a}^{-1} & 0\emx, \quad \text{and}\quad G_B^{-} = \bmx a^{-1} \\ 0\emx,$$
 of $F_A,G_A,F_B$ and $G_B$, respectively, satisfying $F_A^{^-}F_A = I$, $G_A G_A^{^-}=I$ and similarly for $B$, we obtain
 $$A^{^-} = G_A^{^-}F_A^{^-}=\mxl{cc} (b\ovl{b})^{-1} & 0 \\ 0 & 0 \mxr$$
and 
$$B^{-} =G_B^{^-}F_B^{^-}=\mxl{cc} (a\ovl{a})^{-1} & 0 \\ 0 & 0 \mxr.$$

Hence,
\begin{eqnarray*}
 \left[\left(M^2 \right)^*\left(M^2 \right) \right]^{-} \left(M^2 \right)^* &\underset{\text{perm.}}{\approx}& \mxl{cc} r^{-1}\mxl{cc} (b\ovl{b})^{-1} & 0 \\ 0 & 0 \mxr & 0 \\ 0 &  h^{-1}\mxl{cc} (a\ovl{a})^{-1} & 0 \\ 0 & 0 \mxr \mxr \mxl{cc|cc}
\ovl{a}\ovl{b} & \mathit{\ovl{b}w^*} & 0 & 0 \\
\bar{a}\bar{z} & \ovl{z}w^* & 0 & 0 \\
\hline
 0 & 0 & \ovl{a}\ovl{b} & \ovl{a}y^* \\
0 & 0 & \ovl{b}\ovl{x} &  \ovl{x}y^*
\mxr\\
&=& \mxl{cc} r^{-1}\mxl{cc} \ovl{a}b^{-1} & b^{-1}w^* \\ 0 & 0 \mxr & 0 \\ 0 &  h^{-1}\mxl{cc} a^{-1}\ovl{b} & a^{-1}y^* \\ 0 & 0 \mxr \mxr,
\end{eqnarray*}
which implies $  \mxl{cc|cc}
(br)^{-1}\ovl{a} & 0 & 0 & (br)^{-1}w^* \\
0 & 0 & 0 & 0 \\
\hline
 0 & (ah)^{-1}y^* & (ah)^{-1}\ovl{b} & 0 \\
0 & 0 & 0 &  0
\mxr \in  M^2\{1,3\}. $

We can now compute the core EP inverse of $M$ using Lemma \ref{CoreEP1,2}, since
$M^{\coreEP} = M^D M^2 (M^2)^{(1,3)}$. Using the explicit expressions of $M^D$ given in
\eqref{CaseI_Dinv}, of $M^2$, and of $(M^2)^{(1,3)}$, when $x^T y = 0 = z^T w$ and $r,h\neq 0$,
\begin{align*}
M^{\coreEP} &=  \mxl{cc|cc}
0 & h^{-1}y^* & \ovl{b}h^{-1} & 0 \\
\ovl{a}(br)^{-1}y & 0 & 0 & (br)^{-1}yw^* \\
\hline
 \ovl{a}r^{-1} & 0 & 0 & r^{-1}w^* \\
0 & (ah)^{-1}wy^* & \ovl{b}(ah)^{-1}w &  0
\mxr.
\end{align*}

An analogous argument applies to the computation of the \emph{dual core EP inverse} of $M$, which is obtained by analysing the $(1,4)$--invertibility of $M^2$. To this end, we consider the product
$$M^2 \left(M^2 \right)^* =  \mxl{cc|cc} a\ovl{a}p & 0 & 0 & apw^{*} \\
0 & qyy^{*} & \ovl{b}qy & 0 \\ \hline
0 & qby^{*} & b\ovl{b}q & 0 \\
\ovl{a}pw & 0 & 0 & pww^{*}\mxr,$$
where, as in \eqref{eq:all_scalar}, the scalars are given by $p=b\ovl{b}+z^\ast z$ and $q =a\ovl{a} + x^\ast x$.

As before, this matrix admits a convenient block diagonal representation up to permutation
similarity. More precisely,
$$M^2 \left(M^2 \right)^* \underset{\text{perm.}}{\approx}\mxl{cc} pC & 0 \\ 0 & qD \mxr,$$
where
$$C= \mxl{cc} a\ovl{a} & aw^{*} \\ \ovl{a}w & ww^{*} \mxr$$ and $$D= \mxl{cc}b\ovl{b} & by^{*} \\ \ovl{b}y & yy^{*} \mxr.$$
By Proposition \ref{eq:13-14-canonical}, the condition $M^2\{1, 4\}\ne \varnothing$ is equivalent to $\operatorname{rank}\left (M^2(M^2)^*\right) = \operatorname{rank}(M^2)=2$, which in turn holds if and only if $p, q \in \FF\setminus\{0\}$. Under these assumptions, it follows that $\left(M^2 \right)^*\left[ M^2\left(M^2 \right)^* \right]^{-}\in M^2\{1, 4\}$.

To construct an explicit von Neumann inverse of $M^2(M^2)^*$, we observe that such inverse may be chosen, up to permutation similarity, in the block diagonal form
\[
\bigl[M^2(M^2)^*\bigr]^{-}
\;\underset{\mathrm{perm.}}{\approx}\;
\mxl{cc}
p^{-1}C^{-} & 0 \\
0 & q^{-1}D^{-}
\mxr,
\]
where $C^{-}$ and $D^{-}$ are von Neumann inverses of $C$ and $D$, respectively.
We therefore proceed by constructing suitable von Neumann inverses for these two matrices.

Observe that $C$ admits the rank-one factorization
\[
C=
\mxl{c}
a \\ w
\mxr
\mxl{cc}
\ovl{a} & w^{*}
\mxr
=F_C G_C.
\]
We may choose von Neumann inverses
\[
F_C^{-}=
\mxl{cc}
a^{-1} & 0
\mxr
\quad \text{and}\quad
G_C^{-}=
\mxl{c}
\ovl{a}^{-1} \\ 0
\mxr
\]
of $F_C$ and $G_C$, respectively, which satisfy $F_C^{-}F_C=I$ and $G_C G_C^{-}=I$.
It follows that
\[
C^{-}=G_C^{-}F_C^{-}=
\mxl{cc}
(a\ovl{a})^{-1} & 0 \\
0 & 0
\mxr 
\]
is a von Neumann inverse of $C$.

An entirely analogous argument applies to the matrix $D$. Indeed, $D$ admits the full rank factorization
\[
D=
\mxl{c}
b \\ y
\mxr
\mxl{cc}
\ovl{b} & y^{*}
\mxr
=F_D G_D.
\]
We may select
\[
F_D^{-}=
\mxl{cc}
b^{-1} & 0
\mxr
\quad \text{and}\quad
G_D^{-}=
\mxl{c}
\ovl{b}^{-1} \\ 0
\mxr
\]
as von Neumman inverses of $F_D$ and$G_D$, respectively, satisfying $F_D^{-}F_D=I$ and $G_D G_D^{-}=I$. Consequently,
\[
D^{-}=G_D^{-}F_D^{-}=
\mxl{cc}
(b\ovl{b})^{-1} & 0 \\
0 & 0
\mxr
\]
is a von Neumman inverse of $D$.

Therefore,
\begin{eqnarray*}
\left(M^2 \right)^*\left[\left(M^2 \right) \left(M^2 \right)^* \right]^{-} &\underset{\text{perm.}}{\approx} &\mxl{cc|cc}
\ovl{a}\ovl{b} & \mathit{\ovl{b}w^*} & 0 & 0 \\
\bar{a}\bar{z} & \ovl{z}w^* & 0 & 0 \\
\hline
 0 & 0 & \ovl{a}\ovl{b} & \ovl{a}y^* \\
0 & 0 & \ovl{b}\ovl{x} &  \ovl{x}y^*
\mxr \mxl{cc} p^{-1}\mxl{cc} (a\ovl{a})^{-1} & 0 \\ 0 & 0 \mxr & 0 \\ 0 &  q^{-1}\mxl{cc} (b\ovl{b})^{-1} & 0 \\ 0 & 0 \mxr \mxr\\
&=& \mxl{cc} p^{-1}\mxl{cc} a^{-1}\ovl{b} & 0 \\ a^{-1}\ovl{z} & 0 \mxr & 0 \\ 0 &  q^{-1}\mxl{cc} b^{-1}\ovl{a} & 0 \\ b^{-1}\ovl{x} & 0 \mxr \mxr,
\end{eqnarray*}
which implies $$
 \mxl{cc|cc}
(ap)^{-1}\ovl{b} & 0 & 0 & 0 \\
0 & 0 & (bq)^{-1}\ovl{x} & 0 \\
\hline
 0 & 0 & (bq)^{-1}\ovl{a} & 0 \\
(ap)^{-1}\ovl{z} & 0 & 0 &  0
\mxr \in M^2\{1, 4\}.$$

We are now in position to compute the dual core EP inverse of $M$.
By Lemma \ref{CoreEP1,2}, since $i(M)=2$, $M_{_{\coreEP}} = (M^2)^{(1,4)} M^2 M^{D}$.

Using the explicit expressions of $(M^2)^{(1,4)}$, $M^2$, and $M^D$ given in
\eqref{CaseI_Dinv}, when $x^Ty=z^Tw=0$ and $p,q\neq 0$, a straightforward block multiplication yields
\begin{align*}
M_{_{\coreEP}} 
&= \mxl{cc|cc}
0 & \ovl{b}(ap)^{-1}x^T  & \ovl{b}p^{-1} & 0\\
q^{-1}\ovl{x} & 0 & 0 & (bq)^{-1}\ovl{x}z^T \\ \hline
\ovl{a}q^{-1} & 0 & 0 & \ovl{a}(bq)^{-1}z^T \\
0 & (ap)^{-1}\ovl{z}x^T  & p^{-1}\ovl{z} & 0
\mxr.
\end{align*}

Summarizing the above analysis, we obtain the following result.

\begin{thm}
\label{thm:core-EP}
 Let $M$ be a matrix  with associated double star digraph $S_{(m+1),(n+1)}$ of the form (\ref{eq:double-star-canonical}) with $x^T y=z^T w=0$.
 Then
 \ben
 \item $M$ is core EP invertible if and only if $$r= a\ovl{a}+w^{*}w \ne 0 \text{ and } h = b\ovl{b} + y^{*}y \ne 0.$$
 In this case,
 $$M^{\coreEP} =\mxl{cc|cc}
0 & h^{-1}y^* & \ovl{b}h^{-1} & 0 \\
\ovl{a}(br)^{-1}y & 0 & 0 & (br)^{-1}yw^* \\
\hline
 \ovl{a}r^{-1} & 0 & 0 & r^{-1}w^* \\
0 & (ah)^{-1}wy^* & \ovl{b}(ah)^{-1}w &  0
\mxr.$$

 \item $M$ is dual core EP invertible if and only if $$p=b\ovl{b}+ z^*z  \ne 0 \text{ and }  q =a\ovl{a} + x^*x  \ne 0.$$
 In this case,
 $$M_{_{\coreEP}} =\mxl{cc|cc}
0 & \ovl{b}(ap)^{-1}x^T  & \ovl{b}p^{-1} & 0\\
q^{-1}\ovl{x} & 0 & 0 & (bq)^{-1}\ovl{x}z^T \\ \hline
\ovl{a}q^{-1} & 0 & 0 & \ovl{a}(bq)^{-1}z^T \\
0 & (ap)^{-1}\ovl{z}x^T  & p^{-1}\ovl{z} & 0
\mxr. $$
 \een
\end{thm}

We now turn to the analysis of Case II, where $x^T y \neq 0 = z^T w$ and $\zeta = x^T y + ab\neq 0$. In this case, $i(M)=3$ and the Drazin inverse of $M$ is given explicitly by  \eqref{CaseII_Dinv}.

In order to compute the core EP inverse of $M$, Lemma~\ref{CoreEP1,2} requires the
existence of a $(1,3)$--inverse of $M^3$. We therefore begin by computing $M^3$ and $(M^3)^*$. A direct block multiplication yields
\[
M^3 =
\mxl{cc|cc}
0 & \zeta x^T  & a\zeta & 0 \\
\zeta y & 0 & 0 & ayz^T  \\ \hline
\zeta b & 0 & 0 & abz^T  \\
0 & bwx^T  & abw & 0
\mxr
\]
and
\[
(M^3)^* =
\mxl{cc|cc}
0 & \ovl{\zeta} y^* & \ovl{\zeta}\ovl{b} & 0 \\
\ovl{\zeta}\ovl{x} & 0 & 0 & \ovl{b}\ovl{x}w^* \\ \hline
\ovl{a}\ovl{\zeta} & 0 & 0 & \ovl{a}\ovl{b}w^* \\
0 & \ovl{a}\ovl{z}y^* & \ovl{a}\ovl{b}\ovl{z} & 0
\mxr .
\]

By Proposition~\ref{eq:13-14-canonical}, the existence of a $(1,3)$--inverse of $M^3$ is
equivalent to the condition $\operatorname{rank}\!\bigl((M^3)^*M^3\bigr)= \operatorname{rank}(M^3)=2$. We therefore proceed to analyse the Gram-type product $(M^3)^*M^3$. A straightforward computation shows that
$$(M^3)^*M^3 =
\mxl{cc|cc}
\ovl{\zeta}\zeta h & 0 & 0 & a\ovl{\zeta}h z^T \\
0 & \beta \ovl{x}x^T & a\beta \ovl{x} & 0 \\ \hline
0 & \ovl{a}\beta x^T & a\ovl{a}\beta & 0 \\
\ovl{a}\zeta h \ovl{z} & 0 & 0 & a\ovl{a}h \ovl{z}z^T
\mxr,$$
where, as in \eqref{eq:all_scalar}, $h = b\ovl{b}+y^*y$ and $\beta = \zeta\ovl{\zeta} + b\ovl{b}w^*w$.

Up to permutation similarity, $(M^3)^*M^3$ decomposes into a block diagonal matrix,
$$(M^3)^*M^3
\;\underset{\mathrm{perm.}}{\approx}\;
\mxl{cc}
hA & 0 \\
0 & \beta B
\mxr,
$$
where
\[
A =
\mxl{cc}
\zeta\ovl{\zeta} & a\ovl{\zeta}z^T \\
\ovl{a}\zeta\ovl{z} & a\ovl{a}\,\ovl{z}z^T
\mxr,
\qquad
B =
\mxl{cc}
a\ovl{a} & \ovl{a}x^T \\
a\ovl{x} & \ovl{x}x^T
\mxr .
\]

Since $\operatorname{rank}(M^3)=2$, Proposition~\ref{eq:13-14-canonical} shows that
$M^3\{1,3\}\neq\varnothing$ if and only if $h,\beta \in \FF\setminus\{0\}$. Assuming these conditions, we may choose a von Neumann inverse of $(M^3)^*M^3$ of the
form
\[
\bigl[(M^3)^*M^3\bigr]^{-}
\;\underset{\mathrm{perm.}}{\approx}\;
\mxl{cc}
h^{-1}A^{-} & 0 \\
0 & \beta^{-1}B^{-}
\mxr .
\]

We now construct explicit von Neumann inverses for $A$ and $B$. Consider the full rank
factorizations
\[
A=
\mxl{c}
\ovl{\zeta} \\ \ovl{a}\ovl{z}
\mxr
\mxl{cc}
\zeta & az^T
\mxr
=F_AG_A,
\qquad
B=
\mxl{c}
\ovl{a} \\ \ovl{x}
\mxr
\mxl{cc}
a & x
\mxr
=F_BG_B .
\]
Choosing
\[
F_A^{-}=\bmx \ovl{\zeta}^{-1} & 0 \emx,
\quad
G_A^{-}=\mx{\zeta^{-1} \\ 0}, \quad
F_B^{-}=\bmx \ovl{a}^{-1} & 0 \emx,
\quad \text{and} \quad
G_B^{-}=\mx{a^{-1} \\ 0},
\]
we obtain $F_A^{-}F_A=I$, $G_AG_A^{-}=I$ and similarly for $B$, which yields
\[
A^{-}=G_A^{-}F_A^{-}=
\mxl{cc}
(\zeta\ovl{\zeta})^{-1} & 0 \\
0 & 0
\mxr,
\qquad
B^{-}=G_B^{-}F_B^{-}
\mxl{cc}
(a\ovl{a})^{-1} & 0 \\
0 & 0
\mxr .
\]
Hence,

$\left[(M^3)^*M^3\right]^{^-}(M^3)^{*} \underset{\text{perm.}}{\approx} \mxl{cc} h^{-1} \mxl{cc} (\zeta\ovl{\zeta})^{-1} & 0 \\ 0 & 0 \mxr & 0 \\ 0 & \beta^{-1} \mxl{cc} (a\ovl{a})^{-1} & 0 \\ 0 & 0 \mxr \mxr \mxl{cc|cc}
0 & 0 & \bar{\zeta}\bar{b} & \bar{\zeta}y^* \\
0 & 0 & \ovl{a}\ovl{b}\ovl{z} &\bar{a}\bar{z}y^* \\
\hline
 \bar{a}\ovl{\zeta} & \ovl{a}\ovl{b}w^*  & 0 & 0 \\
\ovl{\zeta}\ovl{x} & \ovl{b}\ovl{x}w^* & 0 & 0
\mxr$

\begin{align*}
    &= \mxl{cc} 0 & (\zeta h)^{-1} \mxl{cc} \bar{b} & y^* \\ 0 & 0 \mxr \\ (a\beta)^{-1} \mx{\ovl{\zeta} & \ovl{b}w^* \\ 0 & 0 } & 0\mxr,
\end{align*}
which implies that $$\mxl{cccc} 0 & (\zeta h)^{-1}y^*  &  (\zeta h)^{-1}\bar{b} & 0  \\ 0 & 0 & 0 & 0 \\ (a \beta)^{-1}\ovl{\zeta} & 0 & 0 & (a\beta)^{-1}\ovl{b}w^* \\ 0 & 0 & 0 & 0\mxr \in M^3\{1,3\}.$$

We may now compute the core EP inverse of $M$. By Lemma~\ref{CoreEP1,2}, since
$i(M)=3$, $M^{\coreEP}=M^D M^3 (M^3)^{(1,3)}$. Using the explicit expression of $M^D$ given in \eqref{CaseII_Dinv} and the explicit expressions of $M^3$ and $(M^3)^{(1,3)}$, when  
$x^Ty\neq 0$, $z^T w=0$, $h\neq 0$ and $\beta\neq 0$, a direct block multiplication yields
\begin{align*}
M^{\coreEP} &= \mxl{cccc} 0 & h^{-1} y^*  & h^{-1}\bar{b} & 0  \\
\beta^{-1}\ovl{\zeta} y & 0 & 0 & \beta^{-1} \bar{b} y w^* \\
\beta^{-1}\ovl{\zeta}b  & 0 & 0 & \beta^{-1}b\bar{b} w^* \\
0 & bh^{-1}\zeta^{-1}wy^* & b\bar{b}h^{-1}\zeta^{-1}w & 0\mxr.
\end{align*}

We now follow the same strategy to determine the \emph{dual core EP inverse} of $M$.
In this case, we focus on the set $M^3\{1,4\}$. To this end, we first compute the product $M^3(M^3)^*$. Using the explicit expressions of $M^3$ and $(M^3)^*$ obtained above, we find
\begin{align*}
M^3(M^3)^*=\mxl{cccc}
\zeta\ovl{\zeta}q & 0 & 0 & \zeta\ovl{b}qw^* \\
0 & \alpha yy^* & \alpha\ovl{b}y & 0 \\
0 & \alpha by^* & \alpha b\ovl{b} & 0 \\
b\ovl{\zeta}qw & 0 & 0 & b\ovl{b}qww^*
\mxr ,
\end{align*}
where, as in \eqref{eq:all_scalar}, $q = a\ovl{a}+x^*x$ and $\alpha = \zeta\ovl{\zeta} + a\ovl{a}z^*z$. As before, this matrix admits a convenient block diagonal decomposition up to permutation
similarity,
\[
M^3(M^3)^*
\;\underset{\mathrm{perm.}}{\approx}\;
\mxl{cc}
qC & 0 \\
0 & \alpha D
\mxr ,
\]
where
\[
C=
\mxl{cc}
\zeta\ovl{\zeta} & \zeta\ovl{b}w^* \\
\ovl{\zeta}bw & b\ovl{b}ww^*
\mxr,
\qquad
D=
\mxl{cc}
b\ovl{b} & by^* \\
\ovl{b}y & yy^*
\mxr .
\]

Proposition~\ref{eq:13-14-canonical} now shows that the existence of a $(1,4)$--inverse
of $M^3$ is equivalent to $\operatorname{rank}\!\bigl(M^3(M^3)^*\bigr)=\operatorname{rank}\!\bigl((M^3)^*\bigr)=2$, which holds if and only if $q,\alpha \in \FF\setminus\{0\}$. Assuming these conditions, we may choose a von Neumann inverse of $M^3(M^3)^*$ given by
\[
\bigl[M^3(M^3)^*\bigr]^{-}
\;\underset{\mathrm{perm.}}{\approx}\;
\mxl{cc}
q^{-1}C^{-} & 0 \\
0 & \alpha^{-1}D^{-}
\mxr .
\]
To construct explicit von Neumann inverses of $C$ and $D$, we consider the full rank factorizations
\[
C=
\mxl{c}
\zeta \\ bw
\mxr
\mxl{cc}
\ovl{\zeta} & \ovl{b}w^*
\mxr
=F_CG_C,
\qquad
D=
\mxl{c}
b \\ y
\mxr
\mxl{cc}
\ovl{b} & y^*
\mxr
=F_DG_D .
\]
Choosing
\[
F_C^{-}=\bmx \zeta^{-1} & 0 \emx,
\quad
G_C^{-}=\mx{\ovl{\zeta}^{-1} \\ 0},
\quad
F_D^{-}=\bmx b^{-1} & 0 \emx
\quad \text{and} \quad
G_D^{-}=\mx{\ovl{b}^{-1} \\ 0},
\]
we obtain $F_C^{-}F_C=I$, $G_CG_C^{-}=I$, and analogously for $D$, which yields
\[
C^{-}=G_C^{-}F_C^{-}=
\mxl{cc}
(\zeta\ovl{\zeta})^{-1} & 0 \\
0 & 0
\mxr,
\qquad
D^{-}=G_D^{-}F_D^{-}=
\mxl{cc}
(b\ovl{b})^{-1} & 0 \\
0 & 0
\mxr .
\]
Consequently,
\begin{align*}(M^3)^*\left[M^3(M^3)^*\right]^{^-} & \underset{\text{perm.}}{\approx} \mxl{cc|cc}
0 & 0 & \ovl{\zeta}\bar{b} & \ovl{\zeta}y^* \\
0 & 0 & \ovl{a}\ovl{b}\ovl{z} &\bar{a}\bar{z}y^* \\
\hline
 \ovl{a}\ovl{\zeta} & \ovl{a}\ovl{b}w^*  & 0 & 0 \\
\ovl{\zeta}\ovl{x} & \ovl{b}\ovl{x}w^* & 0 & 0
\mxr \mxl{cc} q^{-1}\mxl{cc}(\zeta \ovl{\zeta})^{-1} & 0 \\ 0 & 0\mxr & 0 \\ 0 & \alpha^{-1} \mxl{cc}(b\ovl{b})^{-1} & 0 \\ 0 & 0\mxr \mxr \\
& = \mxl{cc|cc} 0 & 0 & (\alpha b)^{-1}\ovl{\zeta} & 0 \\
        0 & 0 & \ovl{a}(\alpha b)^{-1}\ovl{z} & 0 \\ \hline
        (\zeta q)^{-1}\ovl{a} & 0 & 0 & 0 \\
        (\zeta q)^{-1}\ovl{x} & 0 & 0 & 0 \mxr,
\end{align*} and we can assert $$\mxl{cc|cc} 0 & 0 & (\alpha b)^{-1}\ovl{\zeta} & 0 \\
        (\zeta q)^{-1}\ovl{x} & 0 & 0 & 0 \\ \hline
        (\zeta q)^{-1}\ovl{a} & 0 & 0 & 0 \\
        0 & 0 & (\alpha b)^{-1}\bar{a}\ovl{z} & 0 \mxr \in M^3\{1,4\}.$$

We are now in position to compute the dual core EP inverse of $M$.
By Lemma~\ref{CoreEP1,2}, since $i(M)=3$, the dual core EP inverse of $M$ is given by $M_{_{\coreEP}}  = (M^3)^{(1,4)}M^3M^D$. Using the explicit expressions of $(M^3)^{(1,4)}$, $M^3$, and $M^D$ given in \eqref{CaseII_Dinv}, when $x^Ty\neq 0$, $z^Tw=0$, $q\neq 0$ and $\alpha\neq 0$, a direct block
multiplication yields
\begin{align*}
    M_{_{\coreEP}}        &= \mxl{cc|cc} 0 & \ovl{\zeta}\alpha^{-1}x^T  & \ovl{\zeta}\alpha^{-1}a & 0  \\
        q^{-1}\ovl{x} & 0 & 0 &  (q \zeta)^{-1}a\ovl{x} z^T  \\ \hline
        \ovl{a}q^{-1} & 0 & 0 & (q \zeta)^{-1}a\ovl{a} z^T  \\
        0 & \bar{a}\alpha^{-1}\ovl{z}x^T  & a \ovl{a}\alpha^{-1} \ovl{z} & 0 \mxr.
\end{align*}

We have therefore proved the following result.

\begin{thm}\label{thm:core-EP_II}  Let $M$ be a matrix over a field with associated double star digraph $S_{(m+1),(n+1)}$ of the form \eqref{eq:double-star-canonical} with $x^T y\ne 0$, $z^T w=0$ and $\zeta = x^T y+ab \ne 0$.
Then
\ben
\item M is core EP invertible if and only if
$$ h=b\ovl{b}+y^*y \neq 0 \text{ and } \beta = \zeta \ovl{\zeta} + b\ovl{b}w^*w \neq 0 .$$
In this case
$$M^{\coreEP}= \mxl{cc|cc} 0 & h^{-1} y^*  & h^{-1}\bar{b} & 0  \\
\beta^{-1}\ovl{\zeta} y & 0 & 0 & \beta^{-1} \bar{b} y w^* \\ \hline
\beta^{-1}\ovl{\zeta}b  & 0 & 0 & \beta^{-1}b\bar{b} w^* \\
0 & bh^{-1}\zeta^{-1}wy^* & b\bar{b}h^{-1}\zeta^{-1}w & 0\mxr$$

\item  M is dual core EP invertible if and only if
$$q= a\ovl{a}+x^*x \neq 0\text{ and }\alpha= \zeta\ovl{\zeta}+a\ovl{a}z^*z \neq 0.$$
In this case,
$$M_{_{\coreEP}}= \mxl{cc|cc} 0 & \ovl{\zeta}\alpha^{-1}x^T  & \ovl{\zeta}\alpha^{-1}a & 0  \\
        q^{-1}\ovl{x} & 0 & 0 &  (q \zeta)^{-1}a\ovl{x} z^T  \\ \hline
        \ovl{a}q^{-1} & 0 & 0 & (q \zeta)^{-1}a\ovl{a} z^T  \\
        0 & \bar{a}\alpha^{-1}\ovl{z}x^T  & a \ovl{a}\alpha^{-1} \ovl{z} & 0 \mxr.$$
\een
\end{thm}

We finally consider Case III, characterized by the conditions $x^T y \neq 0=z^T w$ and $\zeta = x^T y + ab = 0$. In this situation, it follows from \eqref{CaseIII_Dinv} that $i(M)=5$ and $M^D=0$. Consequently, the matrix $M$ is nilpotent of index $5$. Since $M^5=0$, we have $\operatorname{rank}(M^5)=\operatorname{rank}\bigl((M^5)^*M^5\bigr)=\operatorname{rank}\bigl(M^5(M^5)^*\bigr)=0$.
By Proposition~\ref{eq:13-14-canonical}, this implies that
$(M^5)\{1,3\}\neq\varnothing$ and $(M^5)\{1,4\}\neq\varnothing$. We may therefore apply Lemma \ref{CoreEP1,2}. The matrix $M$ is both core EP invertible and dual core EP invertible. Moreover, in both cases the corresponding inverse is the zero matrix. This completes the analysis of Case III and, consequently, the proof of the results
of this section.


\section{MPCEP and \texorpdfstring{$^\ast\!$}{*}CEPMP Inverses}\label{sec:mpcep-cepmp}

In this section we derive explicit formulas for the Moore--Penrose--core EP (MPCEP) inverse
 and the $^\ast\!$core EP--Moore--Penrose ($^\ast\!$CEPMP) inverse for matrices associated with
double star digraphs. 

Since the existence of MPCEP and $^\ast\!$CEPMP inverses requires, respectively,
the existence of the Moore--Penrose inverse together with core--EP or dual core--EP inverse, the analysis splits according to the three structural cases introduced in Section \ref{sec:prelim}.

Recall that whenever $M^\dagger$ and $M^{\coreEP}$ exist, the MPCEP inverse of $M$ is given by \eqref{eq:def-mpcep}, and that whenever $M^\dagger$ and $M_{\coreEP}$ exist, the $^\ast\!$CEPMP inverse of $M$ is given by \eqref{eq:def-cepmp}.

The proofs of the following lemmas are immediate consequences of these identities, combined with the explicit
expressions of $M^\dagger M$ or $MM^\dagger$ from \eqref{projectors} and of
$M^{\coreEP}$ or $M_{\coreEP}$ from Theorems \ref{thm:core-EP}
and \ref{thm:core-EP_II}.

\begin{lem}\label{MPCEP_CaseI} Let $M$ be a matrix  with associated double star digraph $S_{(m+1),(n+1)}$ of the form (\ref{eq:double-star-canonical}) with $x^T y=z^T w=0$. Assume that $M^\dagger$ exists and that $M$ is core EP
invertible (equivalently, the scalars $s,u,t,v,r,h$ defined in \eqref{eq:all_scalar} are nonzero).
Then the MPCEP inverse of $M$ is given by
$$M^{\dagger,\coreEP} =\mxl{cc|cc}
0 & h^{-1}y^* & \ovl{b}h^{-1} & 0 \\
0  & 0 & 0 & 0 \\ \hline
\ovl{a}r^{-1} & 0 & 0 & r^{-1}w^* \\
0 & 0 & 0 & 0
\mxr.$$
\end{lem}

\begin{lem}\label{MPCEP_CaseII}  Let $M$ be a matrix  with associated double star digraph $S_{(m+1),(n+1)}$ of the form (\ref{eq:double-star-canonical}) satisfying $x^Ty \neq 0$, $z^Tw = 0$ and $\zeta = x^Ty + ab \neq 0$. Assume that $M$ is Moore--Penrose and core EP invertible (equivalently, the scalars $s,u,t,v,h,\beta$ defined in \eqref{eq:all_scalar} are nonzero). Then the MPCEP inverse of $M$ is given by
$$M^{\dagger,\coreEP} =
 \mxl{cc|cc} 0 & h^{-1}y^\ast & \ovl{b}h^{-1} & 0 \\
(s\beta)^{-1}\ovl{\zeta}(x^T y) \ovl{x} & 0 & 0 & (s\beta)^{-1}\ovl{b}(x^T y)\ovl{x}w^* \\ \hline
\beta^{-1}\ovl{\zeta}b & 0 & 0 & \beta^{-1}b\ovl{b}w^*\\
0 & 0& 0 & 0 \mxr.$$
\end{lem}

\begin{lem}\label{CEPMP_CaseI} Let $M$ be a matrix  with associated double star digraph $S_{(m+1),(n+1)}$ of the form (\ref{eq:double-star-canonical}) with $x^T y=z^T w=0$. Assume that $M^\dagger$ and $M_{\coreEP}$ both exist (equivalently, the scalars $s,u,t,v,p,q$ defined in \eqref{eq:all_scalar} are nonzero). Then the $^\ast\!$CEPMP inverse of $M$ is given by
$$M_{\coreEP,\dagger} =\mxl{cc|cc} 0 & 0 & \ovl{b}p^{-1} & 0 \\ 
q^{-1}\ovl{x} & 0 & 0 & 0 \\ \hline
\ovl{a}q^{-1} & 0 & 0 & 0  \\
0 & 0 & p^{-1}\ovl{z} & 0 \mxr$$
\end{lem}

\begin{lem}\label{CEPMP_CaseII} Let $M$ be a matrix  with associated double star digraph $S_{(m+1),(n+1)}$ of the form (\ref{eq:double-star-canonical}) with $x^Ty \neq 0$, $z^Tw = 0$ and $\zeta = x^Ty + ab \neq 0$. Assume that $M^\dagger$ exists and that $M$ is dual core EP
invertible (equivalently, the scalars $s,u,t,v,q,\alpha$ defined in \ref{eq:all_scalar} are nonzero). Then the $^\ast\!$CEPMP inverse of $M$ is given by
$$M_{\coreEP,\dagger} =\mxl{cc|cc} 0 &  \ovl{\zeta}(u\alpha)^{-1}(x^T y)y^\ast & a\ovl{\zeta}{\alpha}^{-1} & 0 \\ 
q^{-1}\ovl{x} & 0 & 0 & 0 \\ \hline
\ovl{a} q^{-1} & 0 & 0 & 0 \\
0 & \ovl{a}(u\alpha)^{-1}(x^T y)\ovl{z}y^* & a\ovl{a}\alpha^{-1}\ovl{z} & 0 \mxr$$
\end{lem}

The four lemmas above provide explicit block formulas for the MPCEP and $^\ast\!$CEPMP inverses in Case I (Lemmas \ref{MPCEP_CaseI} and \ref{CEPMP_CaseI}) and Case II (Lemmas \ref{MPCEP_CaseII} and \ref{CEPMP_CaseII}). In Case III, we have $M^{\coreEP}=0$ and $M_{\coreEP}=0$, and therefore $M^{\dagger,\coreEP}=0$
and $M_{\coreEP,\dagger}=0$.

In both Case I and Case II, the MPCEP and $^\ast\!$CEPMP inverses never coincide for matrices associated with double star digraphs. Indeed, inspection of the above formulas
shows that the equality of these two composite inverses in Case I would force all the vectors $x$, $y$, $z$ and $w$ to be zero, contradicting the defining
double star structure. A similar conclusion holds in Case II.


\section{GDC and GC Inverses}
\label{sec:gdcdc-block}

In this section we derive explicit block formulas for the generalized dual core (GDC)
inverse and the generalized core (GC) inverse of matrices associated with double star
digraphs.

These inverses are not independent notions: they are composite outer inverses obtained
by combining the Moore--Penrose inverse with the (dual) core EP inverse, as introduced
in \cite{Mosic2024}. Consequently, their existence does not introduce new algebraic
constraints beyond those already required for Moore--Penrose and core EP
invertibility.

The proofs of the following lemmas are immediate consequences of
\eqref{eq:gdc-explicit} and \eqref{eq:gc-explicit}, combined with the explicit
formulas for $M^\dagger$ from Lemma \ref{thmMPdoublestar} and for
$M^{\coreEP}$ or $M_{\coreEP}$ from Theorems \ref{thm:core-EP}
and \ref{thm:core-EP_II}.

\begin{lem}\label{GDC_CaseI}  Let $M$ be a matrix  with associated double star digraph $S_{(m+1),(n+1)}$ of the form (\ref{eq:double-star-canonical}) with $x^T y=z^T w=0$. Assume that $M$ is Moore--Penrose and core EP
invertible (equivalently, the scalars $s,u,t,v,r,h$ defined in \eqref{eq:all_scalar} are nonzero).
Then the GDC inverse of $M$ is given by
$$M^{GDC} =\mxl{cc|cc} 0 & \ovl{a} (br)^{-1} x^T & b^{-1} & 0 \\
0 & 0 & 0 & 0 \\ \hline
a^{-1} & 0 & 0 & \ovl{b}(ah)^{-1}z^T \\
0 & 0 & 0 & 0 \mxr.$$
\end{lem}

\begin{lem}\label{GDC_CaseII} Let $M$ be a matrix  with associated double star digraph $S_{(m+1),(n+1)}$ of the form (\ref{eq:double-star-canonical}) satisfying $x^Ty \neq 0$, $z^Tw = 0$ and $\zeta = x^Ty + ab \neq 0$. Assume that $M$ is Moore--Penrose invertible and $M^{\coreEP}$ exists (equivalently, the scalars $s,u,t,v,h,\beta$ defined in \eqref{eq:all_scalar} are nonzero). Then the GDC inverse of $M$ is given by
$$M^{GDC} =
\mxl{cc|cc} 0 & \beta^{-1}\ovl{\zeta}x^T & \beta^{-1}(a\ovl{\zeta}+\ovl{b}w^\ast w) & 0 \\ 
(s\zeta)^{-1}(x^T y)\ovl{x} & 0 & 0 & \ovl{b}(hs\zeta)^{-1}(x^T y) \ovl{x}z^T \\ \hline
b\zeta^{-1} & 0 & 0 & b\ovl{b}(h\zeta)^{-1}z^T \\
0 & 0 & 0 & 0\mxr.$$
\end{lem}

\begin{lem}\label{GC_CaseI} Let $M$ be a matrix  with associated double star digraph $S_{(m+1),(n+1)}$ of the form (\eqref{eq:double-star-canonical}) with $x^T y=z^T w=0$. Assume that $M^\dagger$ exists and that $M$ is dual core EP
invertible (equivalently, the scalars $s,u,t,v,p,q$ defined in \eqref{eq:all_scalar} are nonzero). Then the GC inverse of $M$ is given by
$$M^{GC} =\mxl{cc|cc} 0 & 0 & b^{-1} & 0 \\
\ovl{b}(ap)^{-1}y & 0 & 0 & 0 \\ \hline
a^{-1} & 0 & 0 & 0 \\
0 & 0 & \ovl{a}(bq)^{-1}w & 0 \mxr.$$
\end{lem}

\begin{lem}\label{GC_CaseII} Let $M$ be a matrix  with associated double star digraph $S_{(m+1),(n+1)}$ of the form (\ref{eq:double-star-canonical}) with $x^Ty \neq 0$, $z^Tw = 0$ and $\zeta = x^Ty + ab \neq 0$. Assume that $M$ is Moore--Penrose and dual core EP
invertible (equivalently, the scalars $s,u,t,v,q,\alpha$ defined in \ref{eq:all_scalar} are nonzero). Then the GC inverse of $M$ is given by
$$M^{GC}=\mxl{cc|cc} 0 & (u\zeta)^{-1}(x^T y)y^\ast & a\zeta^{-1} & 0 \\
\alpha^{-1}\ovl{\zeta}y & 0 & 0 & 0 \\ \hline
\alpha^{-1}(b\ovl{\zeta}+\ovl{a}z^\ast z) & 0 & 0 & 0 \\
0 & \ovl{a}(uq\zeta)^{-1}(x^T y)wy^\ast & a\ovl{a}(q\zeta)^{-1}w & 0 
\mxr.$$
\end{lem}

The four lemmas above provide explicit block formulas for the generalized dual core
and generalized core inverses in Case I (Lemmas \ref{GDC_CaseI} and \ref{GC_CaseI}) and Case II (Lemmas \ref{GDC_CaseII} and \ref{GC_CaseII}). As before, Case III
corresponds to a degenerate configuration, and it follows immediately that
$M^{\mathrm{GDC}} = 0$ and $M^{\mathrm{GC}} = 0$ in this case.

As in the analysis of the MPCEP and $^\ast\!$CEPMP inverses, the GDC and GC inverses of matrices associated with double star digraphs never coincide in either Case I or Case II.
Indeed, a direct comparison of the explicit formulas in Case I shows that the
identity $M^{\mathrm{GDC}} = M^{\mathrm{GC}}$ would force $x = y = z = w = 0$, which
contradicts the defining structure of a double star digraph. An analogous
contradiction arises when assuming $M^{\mathrm{GDC}} = M^{\mathrm{GC}}$ in Case II.


\section{Summary of existence conditions}
\label{sec:conditions}

For the reader's convenience, Table \ref{tab:existence-conditions} summarizes the existence
conditions for all generalized inverses of matrices associated with double star digraphs studied in this paper, restricted to the nontrivial situations in which the group inverse does not exist and the Drazin inverse is nonzero. The table collects the conditions obtained in Sections \ref{sec:coreEP}--\ref{sec:gdcdc-block} for Cases I and II, highlighting the dependence of each generalized inverse on the structural parameters defined in \eqref{eq:all_scalar}.

\begin{table}[ht]
\centering
\caption{Existence conditions for the generalized inverses of matrices associated with double star digraphs in the non-group-invertible Cases I and II, expressed in terms of the structural parameters defined in \eqref{eq:all_scalar}\vspace{0.4cm}.
}

\label{tab:existence-conditions}
\begin{tabular}{lcc}
\hline
\textbf{Inverse} & \textbf{Case I} & \textbf{Case II} \\ [2pt]\hline
\( M^\dagger \) & \( s, u, t, v \neq 0 \) & \( s, u, t, v \neq 0 \) \\[2pt]
\( M^{\coreEP} \) & \( r, h \neq 0 \) & \( h, \beta \neq 0 \) \\[2pt]
\( M_{\coreEP} \) & \( p, q \neq 0 \) & \( q, \alpha \neq 0 \) \\[2pt]
\( M^{\dagger,\coreEP} \) & \( s, u, t, v, r, h \neq 0 \) & \( s, u, t, v, h, \beta \neq 0 \) \\[2pt]
\( M_{\coreEP,\dagger} \) & \( s, u, t, v, p, q \neq 0 \) & \( s, u, t, v, q, \alpha \neq 0 \) \\[2pt]
\( M^{\mathrm{GDC}} \) & \( s, u, t, v, r, h \neq 0 \) & \( s, u, t, v, h, \beta \neq 0 \) \\[2pt]
\( M^{\mathrm{GC}} \) & \( s, u, t, v, p, q \neq 0 \) & \( s, u, t, v, q, \alpha \neq 0 \) \\[2pt]
\hline
\end{tabular}
\end{table}

\section*{Acknowledgements}
{This research was partially financed by Portuguese Funds
through FCT (Funda\c c\~ao para a Ci\^encia e a Tecnologia) within the Project UID/00013/2025.}
\section*{Data Availability}
No datasets were generated or analyzed during the current study. All results are theoretical and fully contained within the manuscript.
\section*{Conflict of interest}
The authors declare that they have no conflict of interest.

\bibliographystyle{plain}
\bibliography{BibliographyCoreEP}

\bigskip
\end{document}